\documentclass[12pt,leqno]{article}
\usepackage{amsmath,amssymb,bbm,array}
\usepackage{amsthm}
\usepackage[all,cmtip]{xy}
\def\bm#1{\mathbbm{#1}}
\def\fn#1{\mathop{{\rm #1}\vphantom{\dim}}}

\def\hs#1{\hspace*{#1ex}}

\def\vds{\hbox{\hbox to 2\arraycolsep{\hss\vbox{\vbox to 1.3ex{ 
\vss\hbox{.}\vspace{.33ex}\hbox{.}\vspace{.33ex}\hbox{.}\vspace{.33ex}\hbox{.}\vspace{.33ex}\hbox{.}\vss}}\hss}}}
\def\hds#1{\hdotsfor[-.1]{#1}}  

\makeatletter
\renewcommand{\section}{\@startsection{section}{1}{0mm}{12mm}{5mm}{\raggedright\bf\large}}

\def\@citex[#1]#2{\if@filesw\immediate\write\@auxout{\string\citation{#2}}\fi
  \def\@citea{}\@cite{\@for\@citeb:=#2\do
    {\@citea\def\@citea{\@citesep}\@ifundefined
       {b@\@citeb}{{\bf ?}\@warning
       {Citation `\@citeb' on page \thepage \space undefined}}%
{\csname b@\@citeb\endcsname}}}{#1}}
\def\@citesep{; }
\makeatother

\newtheoremstyle{Kang}{}{}{\itshape}{}{\indent\bf}{}{.6em}{}
\theoremstyle{Kang}
\newtheorem{theorem}{Theorem}[section]

\newtheoremstyle{Kremark}{}{}{}{}{\indent\bf}{}{.6em}{}
\theoremstyle{Kremark}
\newtheorem{defn}[theorem]{Definition}

\allowdisplaybreaks[1]  
\numberwithin{equation}{section}

\title{NOETHER'S PROBLEM FOR \\ $\widehat{S_4}$ AND $\widehat{S_5}$}
\author{Ming-chang Kang$^{a,1}$ and Jian Zhou$^{b,1,2}$ \\[3mm]
\begin{minipage}{16cm} \begin{description} \itemsep=-1pt
\item[] $^{a}$Department of Mathematics and Taida Institute of Mathematical\\ Sciences,
National Taiwan University, Taipei
\item[] $^{b}$School of Mathematical Sciences, Peking University, Beijing
\end{description} \end{minipage}}
\date{}

\begin{document}

\maketitle

\footnote{\hspace*{-7.5mm}
E-mail addresses: kang@math.ntu.edu.tw, zhjn@math.pku.edu.cn.}
\footnote{\hspace*{-6mm}$^1\,$Both authors were partially supported
by National Center for Theoretic Sciences (Taipei Office).}
\footnote{\hspace*{-6mm}$^2\,$The work of this paper was finished
when the second-named author visited National Taiwan University
under the support by National Center for Theoretic Sciences (Taipei
Office).}

\begin{abstract}
{\bf Abstract.} Let $k$ be a field, $G$ be a finite group and
$k(x_g:g\in G)$ be the rational function field over $k$, on which
$G$ acts by $k$-automorphisms defined by $h\cdot x_g=x_{hg}$ for
any $g,h\in G$. Noether's problem asks whether the fixed subfield
$k(G):=k(x_g:g\in G)^G$ is $k$-rational, i.e.\ purely
transcendental over $k$. If $\widehat{S_n}$ is the double cover of
the symmetric group $S_n$, in which the liftings of transpositions
and products of disjoint transpositions are of order $4$, Serre
shows that $\bm{Q}(\widehat{S_4})$ and $\bm{Q}(\widehat{S_5})$ are
not $\bm{Q}$-rational. We will prove that, if $k$ is a field such
that $\fn{char} k \neq 2, 3$, and $k(\zeta_8)$ is a cyclic
extension of $k$, then $k(\widehat{S_4})$ is $k$-rational. If it
is assumed furthermore that $\fn{char}k=0$, then
$k(\widehat{S_5})$ is also $k$-rational.
\end{abstract}

\section{Introduction}

Let $k$ be a field, and $L$ be a finitely generated field extension of $k$.
$L$ is called $k$-rational (or rational over $k$) if $L$ is purely transcendental over $k$,
i.e.\ $L$ is isomorphic to some rational function field over $k$.
$L$ is called stably $k$-rational if $L(y_1,\ldots,y_m)$ is $k$-rational for some $y_1,\ldots,y_m$
which are algebraically independent over $k$.
$L$ is called $k$-unirational if $L$ is $k$-isomorphic to a subfield of some $k$-rational field extension of $k$.
It is easy to see that ``$k$-rational" $\Rightarrow$ ``stably $k$-rational" $\Rightarrow$ ``$k$-unirational".

A notion of retract rationality was introduced by Saltman (see \cite{Sa,Ka2}).
It is known that, if $k$ is an infinite field,
then ``stably $k$-rational" $\Rightarrow$ ``retract $k$-rational" $\Rightarrow$ ``$k$-unirational".

Let $k$ be a field and $G$ be a finite group. Let $G$ act on the
rational function field $k(x_g:g\in G)$ by $k$-automorphisms
defined by $h\cdot x_g=x_{hg}$ for any $g,h\in G$. Denote by
$k(G)$ the fixed subfield, i.e. $k(G)=k(x_g:g\in G)^G$. Noether's
problem asks, under what situation, the field $k(G)$ is
$k$-rational.

Noether's problem is related to the inverse Galois problem and the
existence of generic $G$-Galois extensions over $k$. For the
details, see Swan's survey paper \cite{Sw}. The purpose of this
paper is to study Noether's problem for some double covers of the
symmetric group $S_n$.

It is known that, when $n\ge 4$, there are four different double
covers of $S_n$, i.e.\ groups $G$ satisfying the short exact
sequence $1\to C_2\to G\to S_n\to 1$ (see, for example, \cite[page
653]{Se}).

\begin{defn}[{\cite[p.58, 90; HH, p.18; Kar, p.177-181]{GMS}}] \label{d1.1}
Let $C_2=\{\pm 1\}$ be the cyclic group of order 2.
When $n\ge 4$, the group $\widehat{S_n}$ is the unique central extension of $S_n$ by $C_2$,
i.e.\ $1\to C_2\to \widehat{S_n} \to S_n\to 1$,
satisfying the condition that the transpositions and the product of two disjoint transpositions in $S_n$
lift to elements of order 4 in $\widehat{S_n}$.
On the other hand, the group $\widetilde{S_n}$ is the central extension $1\to C_2\to \widetilde{S_n} \to S_n \to 1$
such that a transposition in $S_n$ lifts to an element of order 2 of $\widetilde{S_n}$,
but a product of two disjoint transpositions in $S_n$ lifts to an element of order 4.
\end{defn}

Note that we follow the notation of $\widehat{S_n}$ and
$\widetilde{S_n}$ adopted by Serre in \cite{GMS}, which are
different from those in \cite{HH}.

Using cohomological invariants and trace forms over $\bm{Q}$,
Serre was able to prove the following theorem.

\begin{theorem}[Serre {\cite[p.90]{GMS}}] \label{t1.2}
Both $\bm{Q}(\widehat{S_4})$ and $\bm{Q}(\widehat{S_5})$ are not retract $\bm{Q}$-rational.
In particular, they are not $\bm{Q}$-rational.
\end{theorem}

In \cite[p.89--90]{GMS}, Serre proves that $\fn{Rat}(G/\bm{Q})$ is
false for $G=\widehat{S_4}$ and $\widehat{S_5}$; actually he
proves a bit more. From Serre's proof it is easy to find that
$\bm{Q}(\widehat{S_4})$ and $\bm{Q}(\widehat{S_5})$ are not
retract $\bm{Q}$-rational (see \cite[Section 1]{Ka2} for the
relationship of the property $\fn{Rat}(G/k)$ and the retract
$k$-rationality of $k(G)$). This is the reason why we formulate
Serre's Theorem in the above version. In fact, Theorem \ref{t1.2}
can be perceived also from Serre's own remark in \cite[p.13,
Remark 5.8]{GMS}.

On the other hand, Plans proved the following result.

\begin{theorem}[Plans \cite{Pl1,Pl2}] \label{t1.3}
\rm{(1)} For any field $k$, $k(\widetilde{S_4})$ is $k$-rational.
Thus, if $k$ is a field with $\fn{char}k=0$, $k(\widetilde{S_5})$
is also $k$-rational.

\rm{(2)} For any field $k$ with $\fn{char} k =0$ such that
$\sqrt{-1} \in k$, both $k(\widehat{S_4})$ and $k(\widehat{S_5})$
are $k$-rational.
\end{theorem}

The main result of this article is the following rationality criterion for $k(\widehat{S_4})$ and $k(\widehat{S_5})$.

\begin{theorem} \label{t1.4}
Let $k$ be a field with $\fn{char}k\ne 2$ or $3$, and $\zeta_8$ be
a primitive $8$-th root of unity in some extension field of $k$.
If $k(\zeta_8)$ is a cyclic extension of $k$, then
$k(\widehat{S_4})$ is $k$-rational; if it is assumed furthermore
that $\fn{char}k=0$, then $k(\widehat{S_5})$ is also $k$-rational.
\end{theorem}

When $k$ is a field with $\fn{char}k = p > 0$ and $p \neq 2$, the
assumption that $k(\zeta_8)$ is a cyclic extension of $k$ is
satisfied automatically.

We don't know whether Theorem \ref{t1.2} is valid for fields $k$
other than the field $\bm{Q}$, for examples, some field $k$
satisfying the condition that $k(\zeta_8)$ is not cyclic over $k$.
On the other hand, Serre shows that $\bm{Q}(G)$ is not retract
$\bm{Q}$-rational if $G$ is any one of the groups $\widehat{S_4}$,
$\widehat{S_5}$, $SL_2(\bm{F}_7)$, $SL_2(\bm{F}_9)$ and the
generalized quaternion group of order 16 (see \cite[p.90, Example
33.27]{GMS}). Besides the cases $\widehat{S_4}$ and
$\widehat{S_5}$ studied in Theorem \ref{t1.4}, it is known that
$k(G)$ is $k$-rational provided that $G$ is the generalized
quaternion group of order 16 and $k(\zeta_8)$ is cyclic over $k$
\cite{Ka1}. We don't know whether analogous results as Theorem
\ref{t1.4} are valid when the groups are $SL_2(\bm{F}_7)$ and
$SL_2(\bm{F}_9)$.

\bigskip

The main idea of the proof of Theorem \ref{t1.4} is by applying
the method of Galois descent, namely we enlarge the field $k$ to
$k(\zeta_8)$ first, solve the rationality of
$k(\zeta_8)(\widehat{S_4})$, and then descend the ground field of
$k(\zeta_8)(\widehat{S_4})$ to $k$. This will finish the proof of
the rationality of $k(\widehat{S_4})$. By Plans's Theorem (see
Theorem \ref{t2.5}), $k(\widehat{S_5})$ is a rational extension of
$k(\widehat{S_4})$. Hence $k(\widehat{S_5})$ is $k$-rational also.

In showing that $k(\zeta_8)(\widehat{S_4})$ is
$k(\zeta_8)$-rational and $k(\widehat{S_4})$ is $k$-rational, we
will construct a 4-dimensional faithful representation $V$ of
$\widehat{S_4}$ defined over the field $k$. Although it is not
very difficult to find such a 4-dimensional representation, it
seems the representation and the idea to find it are not
well-known. Once we have this representation, write
$\pi=\fn{Gal}(k(\zeta_8)/k)$. By Theorem \ref{t2.2} of this paper,
it is easy to see that $k(\widehat{S_4})$ is rational over
$k(\zeta_8)(V)^{\langle\widehat{S_4},\pi\rangle}$. Thus it remains
to prove $k(\zeta_8)(V)^{\langle\widehat{S_4},\pi\rangle}$ is
$k$-rational.

The rationality problem of
$k(\zeta_8)(V)^{\langle\widehat{S_4},\pi\rangle}$ is not an easy
job. It requires special efforts and lots of computations. In
several steps we use computers to facilitate the process of
symbolic computation, because computers can save us from the
laborious manual computation. We emphasize computers play only a
minor role in the above sense; we don't use particular codes of
data bases, e.g.\ GAP etc. We will point out that the first
several steps in proving
$k(\zeta_8)(V)^{\langle\widehat{S_4},\pi\rangle}$ is $k$-rational
are rather similar to those in \cite[Section 5]{KZ}. This is not
surprising because we deal with $\widetilde{S_4}$ in \cite[Section
5]{KZ} and the groups $\widehat{S_4}$ and $\widetilde{S_4}$ have a
common subgroup $\widetilde{A_4}$.

We organize this paper as follows. We recall some preliminaries in
Section 2, which will be used in the proof of Theorem \ref{t1.4}.
In Section 3, several low-dimensional faithful representations of
$\widehat{S_4}$ over a field $k$ with $\fn{char}k\ne 2$ will be
constructed (the reader may find another explicit construction in
\cite[p.177-179]{Kar}). Theorem \ref{t1.4} will be proved in
Section 4. In Section 5 we will consider the rationality problem
of $k(G_n)$ (see Definition \ref{d5.1} for the group $G_n$).

Throughout this article,
whenever we write $k(x_1,x_2,x_3,x_4)$ or $k(x,y)$ without explanation,
it is understood that it is a rational function field over $k$.
We will denote $\zeta_8$ (or simply $\zeta$) a primitive 8-th root of unity.

\section{Preliminaries}

We recall several results which will be used in tackling the rationality problem.

\begin{theorem}[Ahmad, Hajja and Kang {\cite[Theorem 3.1]{AHK}}] \label{t2.1}
Let $L$ be any field, $L(x)$ the rational function field of one
variable over $L$ and $G$ a finite group acting on $L(x)$. Suppose
that, for any $\sigma \in G$, $\sigma(L)\subset L$ and
$\sigma(x)=a_\sigma\cdot x+b_\sigma$ where $a_\sigma,b_\sigma\in
L$ and $a_\sigma\ne 0$. Then $L(x)^G=L^G(f)$ for some polynomial
$f\in L[x]$. In fact, if $m=\min\{\deg g(x): g(x)\in L[x]^G, \deg
g(x)\ge 1\}$, any polynomial $f\in L[x]^G$ with $\deg f=m$
satisfies the property $L(x)^G=L^G(f)$.
\end{theorem}

\begin{theorem}[Hajja and Kang {\cite[Theorem 1]{HK}}] \label{t2.2}
Let $G$ be a finite group acting on $L(x_1,\ldots,x_n)$,
the rational function field of $n$ variables over a field $L$.
Suppose that
\begin{enumerate}
\item[{\rm (i)}] for any $\sigma\in G$, $\sigma(L)\subset L$;
\item[{\rm (ii)}] the restriction of the action of $G$ to $L$ is faithful;
\item[{\rm (iii)}] for any $\sigma\in G$,
\[
\begin{pmatrix} \sigma(x_1) \\ \sigma(x_2) \\ \vdots \\ \sigma(x_n) \end{pmatrix}
=A(\sigma)\cdot \begin{pmatrix} x_1 \\ x_2 \\ \vdots \\ x_n \end{pmatrix}+B(\sigma)
\]
where $A(\sigma)\in GL_n(L)$ and $B(\sigma)$ is a $n\times 1$ matrix over $L$.
\end{enumerate}
\end{theorem}

Then there exist elements $z_1,\ldots,z_n\in L(x_1,\ldots,x_n)$ which are algebraically independent over $L$,
and $L(x_1,\ldots,x_n)=L(z_1,\ldots,z_n)$ so that $\sigma(z_i)=z_i$ for any $\sigma \in G$,
any $1\le i\le n$.

\begin{theorem}[Yamasaki \cite{Ya}] \label{t2.3}
Let $k$ be a field with $\fn{char}k\ne 2$, $a\in k\backslash\{0\}$,
$\sigma$ be a $k$-automorphism of the rational function field $k(x,y)$ defined by $\sigma(x)=a/x$, $\sigma(y)=a/y$.
Then $k(x,y)^{\langle \sigma\rangle}=k(u,v)$ where $u=(x-y)/(a-xy)$, $v=(x+y)/(a+xy)$.
\end{theorem}

\begin{theorem}[Masuda {\cite[Theorem 3; HoK, Theorem 2.2]{Ma}}] \label{t2.4}
Let $k$ be any field, $\sigma$ be a $k$-automorphism of the
rational function field $k(x,y,z)$ defined by $\sigma: x\mapsto
y\mapsto z \break \mapsto x$. Then
$k(x,y,z)^{\langle\sigma\rangle}=k(s_1,u,v)=k(s_3,u,v)$ where
$s_1$, $s_2$, $s_3$ are the elementary symmetric functions of
degree one, two, three in $x$, $y$, $z$ and $u$ and $v$ are
defined as
\begin{align*}
u &= \frac{x^2y+y^2z+z^2x-3xyz}{x^2+y^2+z^2-xy-yz-zx}, \\
v &= \frac{xy^2+yz^2+zx^2-3xyz}{x^2+y^2+z^2-xy-yz-zx}.
\end{align*}
\end{theorem}

\begin{theorem}[Plans {\cite[Theorem 11]{Pl2}}] \label{t2.5}
Let $n\ge 5$ be an odd integer and $k$ be a field with
$\fn{char}k=0$. Then $k(\widehat{S_n})$ is rational over
$k(\widehat{S_{n-1}})$.
\end{theorem}

\begin{theorem}[Kang and Plans {\cite[Theorem 1.9]{KP}}] \label{t2.6}
Let $k$ be any field, $G_1$ and $G_2$ be two finite groups.
If both $k(G_1)$ and $k(G_2)$ are $k$-rational, so is $k(G_1\times G_2)$.
\end{theorem}

\section{Faithful representations of \boldmath{$\widehat{S_4}$}}

In this section and the next section, the field $k$ we consider is
of $\fn{char}k\ne 2$ or $3$. We will denote by
$\zeta_8=(1+\sqrt{-1})/\sqrt{2}$, a primitive 8-th root of unity.

In \cite[p.92]{Sp} a generating set of $\widehat{S_4}$ is given
(where the group is called the binary octahedral group) :
$\widehat{S_4}=\langle a',b,c\rangle$ with relations
${a'}^8=b^4=c^6=1$, $ba'b^{-1}={a'}^{-1}$, $cbc^{-1}={a'}^2$,
$(a'c)^2=-{a'}^2b$ (here $-1$ is the element which is equal to
${a'}^4=b^2=c^3$). Note that we have a short exact sequence of
groups
\[
1\to \{\pm1\} \to \widehat{S_4} \stackrel{p}{\to} S_4\to 1
\]
and $p(a')=(1,2,3,4)$, $p(b)=(1,4)(2,3)$, $p(c)=(1,2,3)$.
Note that $p(ba')=(1,4)(2,3)(1,2,3,4)=(1,3)$.

If $\zeta_8\in k$,
a faithful 2-dimensional representation $\Phi:\widehat{S_4}\to GL_2(k)$ is given in \cite[p.92]{Sp} as follows (we write $\zeta=\zeta_8$),
\begin{gather}
\Phi(a')= \begin{pmatrix} \zeta & 0 \\ 0 & \zeta^7 \end{pmatrix}, \quad
\Phi(b) = \begin{pmatrix} 0 & \sqrt{-1} \\ \sqrt{-1} & 0 \end{pmatrix}, \quad
\Phi(c) = \frac{1}{\sqrt{2}}\begin{pmatrix} \zeta^7 & \zeta^7 \\ \zeta^5 & \zeta \end{pmatrix}. \label{eq3.1}
\end{gather}

Suppose that $\sqrt{2}\in k$ (but it is unnecessary that $\sqrt{-1}\in k$).
We may obtain a 4-dimensional representation $\widehat{S_4}\to GL_4(k)$ by substituting
\[
\begin{pmatrix} 0 & -1 \\ 1 & 0 \end{pmatrix}, \quad \begin{pmatrix} \alpha & 0 \\ 0 & \alpha \end{pmatrix}
\]
for $\sqrt{-1}$, $\alpha$ (where $k_0$ is the prime field of $k$
and $\alpha\in k_0(\sqrt{2})$) in Formula \eqref{eq3.1}. This
process is just an easy application of Weil's restriction
\cite[p.38]{We,Vo}. Thus we get
\begin{equation} \label{eq3.2}
\begin{aligned}
a' &\mapsto \frac{1}{\sqrt{2}} \left(\hs{-1}\begin{array}{cc@{\vds}cc}
1 & -1 & & \\ 1 & 1 & & \\[-3\jot] \hds{4} \\[-1\jot] & & 1 & 1 \\ & & -1 & 1 \end{array}\hs{-1}\right), & \quad
b &\mapsto \begin{pmatrix} & & & -1 \\ & & 1 & \\ & -1 & & \\ 1 & & & \end{pmatrix}, \\
c &\mapsto \frac{1}{2}\left(\begin{array}{@{}rrrr@{}} 1 & 1 & 1 & 1 \\ -1 & 1 & -1 & 1 \\ -1 & 1 & 1 & -1 \\ -1 & -1 & 1 & 1 \end{array}\right).
\end{aligned}
\end{equation}

Similarly, when $\sqrt{-2}\in k$ (but it may happen that $\sqrt{-1}\notin k$),
write $\sqrt{-2}=\sqrt{-1}\cdot \sqrt{2}$.
Thus represent $\sqrt{2}$ as $-\sqrt{-1}\cdot \sqrt{-2}$ and $\zeta=(1+\sqrt{-1})/\sqrt{2}$ becomes $\sqrt{-2}(1-\sqrt{-1})/2$.
Substitute
\[
\begin{pmatrix} 0 & -1 \\ 1 & 0 \end{pmatrix}, \quad \begin{pmatrix} \alpha & 0 \\ 0 & \alpha \end{pmatrix}
\]
for $\sqrt{-1}$, $\alpha$ (where $k_0$ is the prime field of $k$
and $\alpha\in k_0(\sqrt{-2})$) in Formula \eqref{eq3.1}. We get
\begin{equation} \label{eq3.3}
\begin{aligned}
a' &\mapsto \frac{\sqrt{-2}}{2} \left(\hs{-1}\begin{array}{cc@{\vds}cc}
1 & 1 & & \\ -1 & 1 & & \\[-3\jot] \hds{4} \\[-1\jot] & & -1 & 1 \\ & & -1 & -1 \end{array}\hs{-1}\right), & \quad
b &\mapsto \begin{pmatrix} & & & -1 \\ & & 1 & \\ & -1 & & \\ 1 & & & \end{pmatrix}, \\
c &\mapsto \frac{1}{2}\left(\begin{array}{@{}rrrr@{}} 1 & 1 & 1 & 1 \\ -1 & 1 & -1 & 1 \\ -1 & 1 & 1 & -1 \\ -1 & -1 & 1 & 1 \end{array}\right).
\end{aligned}
\end{equation}

By the same way, if $\sqrt{-1}\in k$ (but it may happen that $\sqrt{2}\notin k$), substitute
\[
\begin{pmatrix} 0 & 2 \\ 1 & 0 \end{pmatrix}, \quad \begin{pmatrix} \alpha & 0 \\ 0 & \alpha \end{pmatrix}
\]
for $\sqrt{2}$, $\alpha$ (where $k_0$ is the prime filed of $k$
and $\alpha\in k_0(\sqrt{-1})$) in Formula \eqref{eq3.1}. We get
\begin{align}
a' &\mapsto \left(\hs{-1}\begin{array}{cc@{\vds}cc}
0 & 1+\sqrt{-1} & & \\ (1+\sqrt{-1})/2 & 0 & & \\[-3\jot] \hds{4} \\[-1\jot] & & 0 & 1-\sqrt{-1} \\ & & (1-\sqrt{-1})/2 & 0
\end{array}\hs{-1}\right),  \notag \\
b &\mapsto \left(\hs{-1}\begin{array}{cc@{\vds}cc}
& & \sqrt{-1} & 0 \\ & & 0 & \sqrt{-1} \\[-3\jot] \hds{4} \\[-1\jot] \sqrt{-1} & 0 & & \\ 0 & \sqrt{-1} & &
\end{array}\hs{-1}\right), \label{eq3.4} \\
c &\mapsto \begin{pmatrix}  (1-\sqrt{-1})/2 & 0 & (1-\sqrt{-1})/2 & 0 \\ 0 & (1-\sqrt{-1})/2 & 0 & (1-\sqrt{-1})/2 \\
(-1-\sqrt{-1})/2 & 0 & (1+\sqrt{-1})/2 & 0 \\ 0 & (-1-\sqrt{-1})/2 & 0 & (1+\sqrt{-1})/2  \end{pmatrix}. \notag
\end{align}

Finally, from Formula \eqref{eq3.2} we may get a faithful 8-dimensional representation of $\widehat{S_4}$ into $GL_8(k_0)$
where $k_0$ is the prime field of $k$.
Explicitly, substitute
\[
\begin{pmatrix} 0 & 2 \\ 1 & 0 \end{pmatrix}, \quad \begin{pmatrix} \alpha & 0 \\ 0 & \alpha \end{pmatrix}
\]
for $\sqrt{2}$, $\alpha$ where $\alpha\in k_0$ in Formula
\eqref{eq3.2}. We get
\begin{align}
a' &\mapsto \frac{1}{2} \left(\hs{-1}\begin{array}{cccc@{\vds}cccc}
0 & 2 & 0 & -2 & & & & \\
1 & 0 & -1 & 0 & & & & \\
0 & 2 & 0 & 2 & & & & \\
1 & 0 & 1 & 0 & & & & \\[-3\jot] \hds{8} \\[-1\jot]
& & & & 0 & 2 & 0 & 2 \\
& & & & 1 & 0 & 1 & 0 \\
& & & & 0 & -2 & 0 & 2 \\
& & & & -1 & 0 & 1 & 0
\end{array}\hs{-1}\right), \notag \\
b &\mapsto \left(\hs{-1}\begin{array}{cccc@{\vds}cccc}
& & & & & & -1 & 0\\
& & & & & & 0 & -1\\
& & & & 1 & 0 & & \\
& & & & 0 & 1 & & \\[-3\jot] \hds{8} \\[-1\jot]
& & -1 & 0 & & & & \\
& & 0 & -1 & & & & \\
1 & 0 & & & & & & \\
0 & 1 & & & & & &
\end{array}\hs{-1}\right), \label{eq3.5} \\
c &\mapsto \frac{1}{2} \left(\hs{-1}\begin{array}{cccc@{\vds}cccc}
1 & 0 & 1 & 0 & 1 & 0 & 1 & 0 \\
0 & 1 & 0 & 1 & 0 & 1 & 0 & 1 \\
-1 & 0 & 1 & 0 & -1 & 0 & 1 & 0 \\
0 & -1 & 0 & 1 & 0 & -1 & 0 & 1 \\[-3\jot] \hds{8} \\[-1\jot]
-1 & 0 & 1 & 0 & 1 & 0 & -1 & 0 \\
0 & -1 & 0 & 1 & 0 & 1 & 0 & -1 \\
-1 & 0 & -1 & 0 & 1 & 0 & 1 & 0 \\
0 & -1 & 0 & -1 & 0 & 1 & 0 & 1
\end{array}\hs{-1}\right). \notag
\end{align}

%
%

\section{Proof of Theorem \ref{t1.4}}

By Theorem \ref{t2.5}, in case $\fn{char}k=0$ and it is known that
$k(\widehat{S_4})$ is $k$-rational, it follows immediate that
$k(\widehat{S_5})$ is also $k$-rational. Hence, in proving Theorem
\ref{t1.4}, it suffices to prove the rationality of
$k(\widehat{S_4})$.

By assumption, $k(\zeta_8)$ is a cyclic extension of $k$.
Hence at least one of $\sqrt{-1}$, $\sqrt{2}$, $\sqrt{-2}$ belongs to $k$.

\bigskip

\textit{Case} 1. $\zeta_8\in k$.

Since $\fn{char} k \neq 2$ or $3$, the group algebra
$k[\widehat{S_4}]$ is semi-simple. Hence the 2-dimensional
faithful representation provided by Formula \eqref{eq3.1} can be
embedded into the regular representation whose dual space is
$V_{\fn{reg}}=\oplus_{g\in \widehat{S_4}} k\cdot x(g)$ where
$\widehat{S_4}$ acts on $V_{\fn{reg}}$ by $h\cdot x(g)=x(hg)$ for
any $g,h\in \widehat{S_4}$. Applying Theorem \ref{t2.2}, we find
$k(\widehat{S_4})=k(x(g):\break g\in
\widehat{S_4})^{\widehat{S_4}}$ is rational over
$k(x,y)^{\widehat{S_4}}$ where the actions given by Formula
\eqref{eq3.1} are as follows

\begin{align*}
a' :{}& x\mapsto \zeta x,~ y\mapsto \zeta^7 y, \\
b :{}& x\mapsto \sqrt{-1}y,~ y\mapsto \sqrt{-1}x, \\
c :{}& x\mapsto (\zeta^7 x+\zeta^5 y)/\sqrt{2},~ y\mapsto (\zeta^7 x+\zeta y)/\sqrt{2}.
\end{align*}

Define $z=x/y$. Then $k(x,y)=k(z,x)$.
Apply Theorem \ref{t2.1}.
We get $k(z,x)^{\widehat{S_4}}=k(z)^{\widehat{S_4}}(t)$ for some element $t$ fixed by $\widehat{S_4}$.
The field $k(z)^{\widehat{S_4}}$ is $k$-rational by L\"uroth's Theorem.
Hence $k(z,x)^{\widehat{S_4}}$ and $k(\widehat{S_4})$ are $k$-rational.

\bigskip

\textit{Case} 2. $\sqrt{2}\in k$, but $\sqrt{-1}\notin k$.

We will use the 4-dimensional faithful representation of $\widehat{S_4}$ over $k$ provided by Formula \eqref{eq3.2}.
This representation provides an action of $\widehat{S_4}$ on $k(x_1,x_2,x_3,x_4)$ given by
\begin{equation} \label{eq4.1}
\begin{aligned}
a':{}& x_1\mapsto (x_1+x_2)/\sqrt{2},~ x_2\mapsto (-x_1+x_2)/\sqrt{2},~ x_3\mapsto (x_3-x_4)/\sqrt{2},\\
& x_4\mapsto (x_3+x_4)/\sqrt{2}, \\
b :{}& x_1\mapsto x_4 \mapsto -x_1,~ x_2\mapsto -x_3,~ x_3\mapsto x_2, \\
c :{}& x_1\mapsto (x_1-x_2-x_3-x_4)/2,~ x_2\mapsto (x_1+x_2+x_3-x_4)/2, \\
& x_3\mapsto (x_1-x_2+x_3+x_4)/2,~ x_4\mapsto (x_1+x_2-x_3+x_4)/2.
\end{aligned}
\end{equation}

\bigskip

Step 1.
Apply Theorem \ref{t2.2} and use the same arguments in Case 1.
We find that $k(\widehat{S_4})$ is rational over $k(x_1,x_2,x_3,x_4)^{\widehat{S_4}}$.
It remains to show that $k(x_1,x_2,x_3,x_4)^{\widehat{S_4}}$ is $k$-rational.

\bigskip

Step 2.
Write $\pi=\fn{Gal}(k(\sqrt{-1})/k)=\langle \rho \rangle$ where $\rho(\sqrt{-1})=-\sqrt{-1}$.

We extend the actions of $\pi$ and $\widehat{S_4}$ on $k(\sqrt{-1})$ and $k(x_1,x_2,x_3,x_4)$ to
$k(\sqrt{-1})\break (x_1,x_2,x_3,x_4)$ by requiring that $\rho(x_i)=x_i$ for $1\le i\le 4$ and $g(\sqrt{-1})=\sqrt{-1}$
for all $g\in \widehat{S_4}$.
It follows that $k(x_1,x_2,x_3,x_4)^{\widehat{S_4}}=\{ k(\sqrt{-1})(x_1,x_2,x_3,x_4)^{\langle\rho\rangle}\}^{\langle a',b,c\rangle}
=k(\sqrt{-1})(x_1,x_2,x_3,x_4)^{\langle a',b,c,\rho\rangle}$.

Define $y_1,y_2,y_3,y_4\in k(\sqrt{-1})(x_1,x_2,x_3,x_4)$ by
\begin{align*}
y_1 &= \sqrt{-1} x_1+\sqrt{-1} x_2-x_3+x_4, & y_2 &= -\sqrt{-1}x_1+\sqrt{-1}x_2+x_3+x_4, \\
y_3 &= x_1-x_2-\sqrt{-1} x_3-\sqrt{-1} x_4, & y_4 &= x_1+x_2-\sqrt{-1}x_3+\sqrt{-1}x_4.
\end{align*}

Then $k(\sqrt{-1})(x_1,x_2,x_3,x_4)=k(\sqrt{-1})(y_1,y_2,y_3,y_4)$ and the actions in Formula \eqref{eq4.1} becomes
\begin{equation} \label{eq4.2}
\begin{aligned}
a':{}& y_1\mapsto (y_1+y_2)/\sqrt{2},~ y_2\mapsto (-y_1+y_2)/\sqrt{2},~ y_3\mapsto (y_3+y_4)/\sqrt{2},\\
& y_4\mapsto (-y_3+y_4)/\sqrt{2}, \\
b:{}& y_1 \mapsto \sqrt{-1} y_1,~ y_2\mapsto -\sqrt{-1}y_2,~ y_3\mapsto \sqrt{-1}y_3,~ y_4\mapsto -\sqrt{-1}y_4, \\
c:{}& y_1\mapsto (y_1-\sqrt{-1}y_2)/(1+\sqrt{-1}),~y_2\mapsto (y_1+\sqrt{-1}y_2)/(1+\sqrt{-1}), \\
& y_3\mapsto (y_3-\sqrt{-1}y_4)/(1+\sqrt{-1}),~ y_4\mapsto (y_3+\sqrt{-1}y_4)/(1+\sqrt{-1}), \\
\rho:{}& y_1\mapsto -\sqrt{-1} y_4,~ y_2\mapsto \sqrt{-1} y_3,~ y_3\mapsto \sqrt{-1}y_2,~ y_4\mapsto -\sqrt{-1}y_1.
\end{aligned}
\end{equation}

Note that the action of ${a'}^2$ is given by
\[
{a'}^2: y_1\mapsto y_2\mapsto -y_1,~ y_3\mapsto y_4 \mapsto -y_3.
\]

It may be interesting if the reader is willing to compare the
actions in Formula \eqref{eq4.2} with those in \cite[Section
4]{KZ}. It turns out that the formulae for $b$, ${a'}^2$, $c^2$
are completely the same as those for $\lambda_1$, $\lambda_2$,
$\sigma$ in \cite[Formula (4.3)]{KZ}. As mentioned before, both
the subgroups $\langle b,{a'}^2,c^2\rangle$ and $\langle
\lambda_1,\lambda_2,\sigma\rangle$ are isomorphic to
$\widetilde{A_4}$ ($\widetilde{A_4}=p^{-1}(A_4)$ in the notation
of Section 3) as abstract groups.

\bigskip

Step 3.
Before we find $k(\sqrt{-1})(y_1,y_2,y_3,y_4)^{\langle\widehat{S_4},\pi\rangle}$,
we will find $k(\sqrt{-1})(y_1,y_2,y_3,\break y_4)^{\langle b,{a'}^2\rangle}$ first.
The method is the same as Step 3 and Step 4 in \cite[Section 4]{KZ}.
We will write down the details,
for the convenience of the reader.

Define $z_1=y_1/y_2$, $z_2=y_3/y_4$, $z_3=y_1/y_3$.
By Theorem \ref{t2.1}, we find that $k(\sqrt{-1})\break (y_1,y_2,y_3,y_4)^{\langle\widehat{S_4},\pi\rangle}
=k(\sqrt{-1})(z_1,z_2,z_3)(y_4)^{\langle\widehat{S_4},\pi\rangle}=k(\sqrt{-1})(z_1,z_2,z_3)^{\langle\widehat{S_4},\pi\rangle}(z_0)$
where $z_0$ is fixed by the actions of $\widehat{S_4}$ and $\pi$.
It remains to show that $k(\sqrt{-1})(z_1,z_2,z_3)^{\langle\widehat{S_4},\pi\rangle}$ is $k$-rational.

Define $u_1=z_1/z_2$, $u_2=z_1z_2$, $u_3=z_3$.
Then $k(\sqrt{-1})(z_1,z_2,z_3)^{\langle b\rangle}=k(\sqrt{-1})(u_1,u_2,\break u_3)$.
The action of ${a'}^2$ is given by
\[
{a'}^2: u_1 \mapsto 1/u_1,~ u_2\mapsto 1/u_2, ~ u_3\mapsto u_3/u_1.
\]

Define $v_1=(u_1-u_2)/(1-u_1u_2)$, $v_2=(u_1+u_2)/(1+u_1u_2)$, $v_3=u_3(1+(1/u_1))$.
Then $k(\sqrt{-1})(u_1,u_2,u_3)^{\langle {a'}^2\rangle}=k(\sqrt{-1})(u_1,u_2,v_3)^{\langle{a'}^2\rangle}=k(\sqrt{-1})(v_1,v_2,v_3)$
by Theorem \ref{t2.3} (note that ${a'}^2(v_3)=v_3$).
In summary, $k(\sqrt{-1})(z_1,z_2,z_3)^{\langle b,{a'}^2\rangle}=k(\sqrt{-1})(v_1,v_2,v_3)$.

\bigskip

Step 4.
The action of $c$ on $v_1$, $v_2$, $v_3$ is given by
\[
c: v_1 \mapsto 1/v_2,~ v_2 \mapsto v_1/v_2, ~ v_3\mapsto
v_3(v_1+v_2)/[v_2(1+v_1)].
\]

Define $X_3=v_3(1+v_1+v_2)/[(1+v_1)(1+v_2)]$.
Then $c(X_3)=X_3$ and $k(\sqrt{-1})(v_1,v_2,v_3)\break =k(\sqrt{-1})(v_1,v_2,X_3)$.
Thus we may apply Theorem \ref{t2.4} (regarding $v_1$, $1/v_2$, $v_2/v_1$ as $x$, $y$, $z$ in Theorem \ref{t2.4}).
More precisely, define
\begin{align*}
X_1 &= (v_1^3v_2^3+v_1^3+v_2^3-3v_1^2v_2^2)/(v_1^4v_2^2+v_2^4+v_1^2-v_1^2v_2^3-v_1v_2^2-v_1^3v_2), \\
X_2 &= (v_1v_2^4+v_1v_2+v_1^4v_2-3v_1^2v_2^2)/(v_1^4v_2^2+v_2^4+v_1^2-v_1^2v_2^3-v_1v_2^2-v_1^3v_2).
\end{align*}

By Theorem \ref{t2.4} we get $k(\sqrt{-1})(v_1,v_2,X_3)^{\langle c\rangle}=k(\sqrt{-1})(X_1,X_2,X_3)$.

\bigskip

Step 5.
With the aid of computers, the actions of $a'$ and $\rho$ on $X_1$, $X_2$, $X_3$ are given by
\begin{align*}
a':{}& X_1 \mapsto X_1/(X_1^2-X_1X_2+X_2^2),~ X_2\mapsto X_2/(X_1^2-X_1X_2+X_2^2),~ X_3\mapsto X_3, \\
\rho:{}& X_1\mapsto X_2/(X_1^2-X_1X_2+X_2^2),~ X_2\mapsto X_1/(X_1^2-X_1X_2+X_2^2),~ X_3\mapsto -2A/X_3
\end{align*}
where $A=g_1g_2g_3^{-1}$ and
\begin{align*}
g_1 &= (1+X_1)^2-X_2(1+X_1)+X_2^2, \quad g_2 = (1+X_2)^2-X_1(1+X_2)+X_1^2, \\
g_3 &= 1+X_1+X_2+X_1^3+X_2^3+X_1X_2(3X_1X_2-2X_1^2-2X_2^2+2)+X_1^4+X_2^4.
\end{align*}

Note that $\rho(g_1)=g_2/(X_1^2-X_1X_2+X_2^2)$.

Define $Y_1=X_1/X_2$, $Y_2=X_1$, $Y_3=X_1X_3/g_1$.
We find that
\[
a': Y_1\mapsto Y_1,~ Y_2\mapsto Y_1^2/(Y_2(1-Y_1+Y_1^2)),~ Y_3\mapsto Y_3.
\]

Thus $k(\sqrt{-1})(X_1,X_2,X_3)^{\langle a' \rangle}=k(\sqrt{-1})(Y_1,Y_2,Y_3)^{\langle a'\rangle}=k(\sqrt{-1})(Z_1,Z_2,Z_3)$
where $Z_1=Y_1$, $Z_2=Y_2+a'(Y_2)$, $Z_3=Y_3$.

\bigskip

Step 6.
Using computers, we find that the action of $\rho$ is given by
\[
\rho: Z_1\mapsto 1/Z_1,~ Z_2\mapsto Z_2/Z_1,~ Z_3\mapsto
-2Z_1^3/(A'Z_3)
\]
where
$A'=-2Z_1^2+Z_1Z_2+Z_2^2+4Z_1^3-2Z_1Z_2^2-2Z_1^4+3Z_1^2Z_2^2+Z_1^4Z_2-2Z_1^3Z_2^2+Z_1^4Z_2^2$.

Define $U_1=Z_2+\rho(Z_2)$, $U_2=\sqrt{-1}(Z_2-\rho(Z_2))$, $U_3=Z_3+\rho(Z_3)$, $U_4=\sqrt{-1}(Z_3-\rho(Z_3))$.
It is easy to verify that $k(\sqrt{-1})(Z_1,Z_2,Z_3)^{\langle\rho\rangle}=k(U_1,U_2,U_3,U_4)$ with a relation
\[
U_3^2+U_4^2+32(U_1^2+U_2^2)/B=0
\]
where $B=(U_1^2-3U_2^2)^2+4U_1(U_1^2-3U_2^2)+32U_2^2$.

Divide the above relation by $16(U_1^2+U_2^2)^2/B^2$. We get
\[
\bigl(BU_3/(4U_1^2+4U_2^2)\bigr)^2+\bigl(BU_4/(4U_1^2+4U_2^2)\bigr)^2+2B/(U_1^2+U_2^2)=0.
\]

Multiply this relation by $U_1^2+U_2^2$ and use the identity $(\alpha^2+\beta^2)(\gamma^2+\delta^2)
=(\alpha\delta+\beta\gamma)^2+(\alpha\gamma-\beta\delta)^2$.
The relation is simplified as
\begin{equation} \label{eq4.3}
V_3^2+V_4^2+2B=0
\end{equation}
where $V_3=B(U_1U_3+U_2U_4)/(4U_1^2+4U_2^2)$,
$V_4=B(U_1U_4-U_2U_3)/(4U_1^2+4U_2^2)$. Note that
$k(U_1,U_2,U_3,U_4)=K(U_1,U_2,V_3,V_4)$.

Define $w_1=8U_1/(U_1^2-3U_2^2)$, $w_2=8U_2/(U_1^2-3U_2^2)$, $w_3=V_3/(U_1^2-3U_2^2)$, $w_4=V_4/(U_1^2-3U_2^2)$.
Then $k(U_1,U_2,V_3,V_4)=k(w_1,w_2,w_3,w_4)$ and the relation \eqref{eq4.3} becomes
\[
w_3^2+w_4^2+2+w_1+w_2^2=0.
\]

Hence $w_1\in k(w_2,w_3,w_4)$.
Thus $k(\sqrt{-1})(Z_1,Z_2,Z_3)^{\langle\rho\rangle}=k(w_2,w_3,w_4)$ is $k$-rational. Done.

\bigskip

\textit{Case} 3. $\sqrt{-2}\in k$, but $\sqrt{-1}\notin k$.

We use the 4-dimensional faithful representation of $\widehat{S_4}$ over $k$ provided by Formula \eqref{eq3.3}.
This representation provides an action of $\widehat{S_4}$ on $k(x_1,x_2,x_3,x_4)$ given by
\begin{align*}
a':{}& x_1\mapsto \sqrt{-2}(x_1-x_2)/2,~ x_2\mapsto \sqrt{-2}(x_1+x_2)/2,~ x_3\mapsto \sqrt{-2}(-x_3-x_4)/2, \\
& x_4\mapsto \sqrt{-2}(x_3-x_4)/2, \\
b:{}& x_1\mapsto x_4\mapsto -x_1,~ x_2\mapsto -x_3,~ x_3\mapsto x_2, \\
c:{}& x_1\mapsto (x_1-x_2-x_3-x_4)/2,~ x_2\mapsto (x_1+x_2+x_3-x_4)/2, \\
& x_3\mapsto (x_1-x_2+x_3+x_4)/2,~ x_4\mapsto (x_1+x_2-x_3+x_4)/2.
\end{align*}

The proof of this case is very similar to that of Case 2.

\bigskip

Step 1.
Apply Theorem \ref{t2.2}.
We find that $k(\widehat{S_4})$ is rational over $k(x_1,x_2,x_3,x_4)^{\widehat{S_4}}$.
Hence the proof is reduced to proving $k(x_1,x_2,x_3,x_4)^{\widehat{S_4}}$ is $k$-rational.

\bigskip

Step 2.
Write $\pi=\fn{Gal}(k(\sqrt{-1})/k)=\langle\rho\rangle$ where $\rho(\sqrt{-1})=-\sqrt{-1}$.

Extend the actions of $\pi$ and $\widehat{S_4}$ to $k(\sqrt{-1})(x_1,x_2,x_3,x_4)$ as in Step 2 of Case 2.
We find that $k(x_1,x_2,x_3,x_4)^{\widehat{S_4}}=k(\sqrt{-1})(x_1,x_2,x_3,x_4)^{\langle a',b,c,\rho\rangle}$.

Define $y_1,y_2,y_3,y_4\in k(\sqrt{-1})(x_1,x_2,x_3,x_4)$ by
\begin{align*}
y_1 &= -x_1-\sqrt{-1}x_2+x_3+\sqrt{-1}x_4, & y_2 &= \sqrt{-1}x_1-x_2+\sqrt{-1}x_3-x_4, \\
y_3 &= x_1-\sqrt{-1}x_2+x_3-\sqrt{-1}x_4, & y_4 &= \sqrt{-1}x_1+x_2-\sqrt{-1}x_3-x_4.
\end{align*}

We get $k(\sqrt{-1})(x_1,x_2,x_3,x_4)=k(\sqrt{-1})(y_1,y_2,y_3,y_4)$ and the actions are
\begin{equation} \label{eq4.4}
\begin{aligned}
a':{}& y_1\mapsto (-y_1-y_2)/\sqrt{2},~ y_2\mapsto (y_1-y_2)/\sqrt{2},~ y_3\mapsto (y_3+y_4)/\sqrt{2}, \\
& y_4\mapsto (-y_3+y_4)/\sqrt{2}, \\
b:{}& y_1\mapsto \sqrt{-1}y_1,~y_2\mapsto -\sqrt{-1}y_2,~ y_3\mapsto \sqrt{-1}y_3,~ y_4\mapsto -\sqrt{-1}y_4, \\
c:{}& y_1\mapsto (y_1-\sqrt{-1}y_2)/(1+\sqrt{-1}),~ y_2\mapsto (y_1+\sqrt{-1}y_2)/(1+\sqrt{-1}), \\
& y_3\mapsto (y_3-\sqrt{-1}y_4)/(1+\sqrt{-1}), ~ y_4\mapsto (y_3+\sqrt{-1}y_4)/(1+\sqrt{-1}), \\
\rho:{}& y_1\mapsto \sqrt{-1}y_4,~ y_2\mapsto -\sqrt{-1}y_3,~ y_3\mapsto -\sqrt{-1}y_2,~ y_4\mapsto \sqrt{-1}y_1.
\end{aligned}
\end{equation}

Note that the action of ${a'}^2$ is
\[
{a'}^2: y_1\mapsto y_2\mapsto -y_1,~ y_3\mapsto y_4\mapsto -y_3.
\]

Compare Formula \eqref{eq4.2} and Formula \eqref{eq4.4}. The
actions of ${a'}^2$, $b$, $c$ in both cases are the same.

\bigskip

Step 3.
Define $z_1=y_1/y_2$, $z_2=y_3/y_4$, $z_3=y_1/y_3$.
As in Step 3 of Case 2, it suffices to prove $k(\sqrt{-1})(z_1,z_2,z_3)^{\langle\widehat{S_4},\pi\rangle}$ is $k$-rational.

Define $u_1$, $u_2$, $u_3$, $v_1$, $v_2$, $v_3$, $X_1$, $X_2$, $X_3$ by the same formulae as in Step 3 and Step 4 of Case 2.
We find that $k(\sqrt{-1})(z_1,z_2,z_3)^{\langle b,{a'}^2,c\rangle}=k(\sqrt{-1})(X_1,X_2,X_3)$.

\bigskip

Step 4. The actions of $a'$, $\rho$ on $X_1$, $X_2$, $X_3$ are
slightly different from Step 5 of Case 2. In the present case, we
have
\begin{align*}
a':{}& X_1\mapsto X_1/(X_1^2-X_1X_2+X_2^2),~X_2\mapsto X_2/(X_1^2-X_1X_2+X_2^2),~ X_3\mapsto -X_3, \\
\rho:{}& X_1\mapsto X_2/(X_1^2-X_1X_2+X_2^2),~ X_2\mapsto X_1/(X_1^2-X_1X_2+X_2^2),~X_3 \mapsto -2A/X_3
\end{align*}
where $A=g_1g_2g_3^{-1}$ and
\begin{align*}
g_1 &= (1+X_1)^2-X_2(1+X_1)+X_2^2, \quad g_2= (1+X_2)^2-X_1(1+X_2)+X_1^2, \\
g_3 &= 1+X_1+X_2+X_1^3+X_2^3+X_1X_2(3X_1X_2-2X_1^2-2X_2^2+2)+X_1^4+X_2^4.
\end{align*}

Note that the action of $\rho$ is the same as in Step 5 of Case 2.

Define $Y_1=X_1/X_2$, $Y_2=X_1$, $Y_3=X_1X_3/g_1$. We get
\[
a': Y_1\mapsto Y_1,~ Y_2\mapsto Y_1^2/\bigl(Y_2(1-Y_1+Y_1^2)\bigr),~ Y_3\mapsto -Y_3.
\]

Thus $k(\sqrt{-1})(X_1,X_2,X_3)^{\langle a'\rangle}=k(\sqrt{-1})(Y_1,Y_2,Y_3)^{\langle a'\rangle}=k(\sqrt{-1})(Z_1,Z_2,Z_3)$
where $Z_1=Y_1$, $Z_2=Y_2+a'(Y_2)$, $Z_3=Y_3(Y_2-a'(Y_2))$.

\bigskip

Step 5.
Using computers, we find that the action of $\rho$ is given by
\[
\rho: Z_1\mapsto 1/Z_1,~ Z_2\mapsto Z_2/Z_1,~ Z_3\mapsto C/Z_3
\]
where $C=2Z_1^2(-4Z_1^2+Z_2^2-Z_1Z_2^2+Z_1^2Z_2^2)/[(1-Z_1+Z_1^2)(-2Z_1^2+Z_1Z_2+Z_2^2+4Z_1^3-
2Z_1Z_2^2-2Z_1^4+3Z_1^2Z_2^2+Z_1^4Z_2-2Z_1^3Z_2^2+Z_1^4Z_2^2)]$.

Define $U_1=Z_2+\rho(Z_2)$, $U_2=\sqrt{-1}(Z_2-\rho(Z_2))$, $U_3=Z_3+\rho(Z_3)$, $U_4=\sqrt{-1}(Z_3-\rho(Z_3))$.
We find that $k(\sqrt{-1})(Z_1,Z_2,Z_3)^{\langle\rho\rangle}=k(U_1,U_2,U_3,U_4)$ with a relation
\begin{equation} \label{eq4.5}
U_3^2+U_4^2=8(U_1^2+U_2^2)^2(-16+U_1^2-3U_2^2)/\bigl(B(U_1^2-3U_2^2)\bigr)
\end{equation}
where $B=(U_1^2-3U_2^2)^2+4U_1(U_1^2-3U_2^2)+32 U_2^2$.

Note that the above formula of $B$ is identically the same as that in Step 6 of Case 2.

It remains to simplify the relation in Formula \eqref{eq4.5}.

Divide both sides of Formula \eqref{eq4.5} by $(U_1^2+U_2^2)^2$. We get
\[
\bigl(U_3/(U_1^2+U_2^2)\bigr)^2+\bigl(U_4/(U_1^2+U_2^2)\bigr)^2=8(-16+U_1^2-3U_2^2)/\bigl(B(U_1^2-3U_2^2)\bigr).
\]

Divide both sides of the above identity by $(2(U_1^2-3U_2^2)/B)^2$.
We get a relation
\begin{equation} \label{eq4.6}
V_3^2+V_4^2=2(1-V_1^2+3V_2^2)(1+V_1+2V_2^2)
\end{equation}
where $V_1=4U_1/(U_1^2-3U_2^2)$, $V_2=4U_2/(U_1^2-3U_2^2)$, $V_3=BU_3/((U_1^2-3U_2^2)(2U_1^2+2U_2^2))$,
$V_4=BU_4/((U_1^2-3U_2^2)(2U_1^2+2U_2^2))$.

Note that $k(U_1,U_2,U_3,U_4)=k(V_1,V_2,V_3,V_4)$.

Define $w_1=1/(1+V_1)$, $w_2=V_2/(1+V_1)$, $w_3=V_3/(1+V_1)^2$, $w_4=V_4/(1+V_1)^2$.
We get $k(V_1,V_2,V_3,V_4)=k(w_1,w_2,w_3,w_3)$ and the relation \eqref{eq4.6} becomes
\[
w_3^2+w_4^2=2(-1+2w_1+3w_2^2)(w_1+2w_2^2).
\]

Divide the above identity by $(w_1+2w_2^2)^2$.
We get
\[
\big(w_3/(w_1+2w_2^2)\big)^2+\big(w_4/(w_1+2w_2^2)\big)^2=2(-1+2w_1+3w_2^2)/(w_1+2w_2^2).
\]

Since $2(-1+2w_1+3w_2^2)/(w_1+2w_2^2)$ is a ``fractional linear transformation" of $w_1$
and it belongs to $k(w_2,w_3/(w_1+2w_2^2),w_4/(w_1+2w_2^2))$,
we find $w_1\in k(w_2,w_3/(w_1+2w_2^2),w_4/(w_1+2w_2^2))$.
Thus $k(w_1,w_2,w_3,w_4)=k(w_2,w_3/(w_1+2w_2^2),w_4/(w_1+2w_2^2))$.
We find that $k(\sqrt{-1})(Z_1,Z_2,Z_3)^{\langle\rho\rangle}$ is $k$-rational.

\bigskip

\textit{Case} 4. $\sqrt{-1}\in k$, but $\sqrt{2}\notin k$.

The proof is similar to that in Case 2 or Case 3; thus the
detailed proof is omitted. In case $\fn{char} k =0$, we may apply
Plans's Theorem, i.e. Theorem \ref{t1.3}.  \qed

\section{Other double covers of \boldmath{$S_n$}}

In this section we consider the rationality problem of $G_n$ which
is a double cover of the symmetric group other than
$\widehat{S_n}$ and $\widetilde{S_n}$.

There are four double covers of the symmetric group $S_n$ when $n
\geq 4$. The trivial case is the split group $S_n\times C_2$. The
rationality problem of the group $S_n\times C_2$ is easy because
we may apply Theorem \ref{t2.6}. It remains to consider the
non-split cases: They are $\widehat{S_n}$, $\widetilde{S_n}$, and
the group $G_n$ defined below.

\bigskip
\begin{defn} \label{d5.1}
For $n\ge 3$, consider the group $G_n$ such that the short
exact sequence $1\to \{\pm 1\}\to G_n \stackrel{p}{\to} S_n\to 1$
is induced by the cup product $\varepsilon_n\cup \varepsilon_n\in
H^2(S_n,\{\pm 1\})$ (see, for example, \cite[page 654]{Se}) where
$\varepsilon_n: S_n\to \{\pm 1\}$ is the signed map, i.e.\
$\varepsilon_n(\sigma)=-1$ if and only if $\sigma\in S_n$ is an
odd permutation. Note that the group $G_n$ is denoted by
$\overline{S_n}$ in \cite{Pl2}.
\end{defn}

The group $G_n$ can be constructed explicitly as follows. Let
$1\to \{\pm 1\}\to C_4=\{\pm\sqrt{-1},\pm 1\}\stackrel{p_0}{\to}
\{\pm 1\}\to 1$ be the short exact sequence defined by
$p_0(\sqrt{-1})=-1$. The group $G_n$ can be realized as the
pull-back of the following diagram
\[
\xymatrix{ & S_n \ar[d]^{\varepsilon_n} \\ C_4 \ar[r]_{\!\!\!p_0} & \{\pm 1\}.}
\]

Explicitly, as a subgroup of $S_n\times C_4$, $G_n=
\{(\sigma,(\sqrt{-1})^i)\in S_n\times C_4: \epsilon_n(\sigma) =
p_0((\sqrt{-1})^i) \} =(A_n\times \{\pm 1\})\cup
\{(\sigma,\pm\sqrt{-1})\in S_n\times C_4: \sigma\notin A_n\}$.

If $k$ is a field with $\fn{char}k\ne 2$, a faithful
$2n$-dimensional representation can be defined as follows. Let
$X=\bigl(\oplus_{1\le i\le n} k\cdot x_i\bigr) \oplus
\bigl(\oplus_{1\le i\le n} k\cdot y_i\bigr)$ and $G_n$ acts on $X$
by, for $1\le i\le n$,
\begin{align}
t:{} & x_i \mapsto -x_i,~ y_i\mapsto -y_i, \notag \\
\tau:{} & x_i\mapsto x_{\tau(i)},~ y_i\mapsto y_{\sigma^{-1}\tau\sigma(i)} \label{eq5.1} \\
\bar{\sigma}:{} & x_i\mapsto y_i \mapsto -x_i \notag
\end{align}
where $t=(1,-1)\in G_n \subset S_n\times C_4$, $\tau\in A_n$ and $\tau$ is identified with $(\tau,1)\in G_n$,
$\sigma=(1,2)\in S_n$ and $\bar{\sigma}=(\sigma,\sqrt{-1})\in G_n$.

The following theorem was proved by Plans \cite[Theorem 14
(b)]{Pl2} under the assumptions that $\fn{char} k =0$ and
$\sqrt{-1}\in k$. Our proof is different from Plans's proof even
in the situation $\fn{char} k =0$.

\begin{theorem} \label{t5.2}
Assume that $k$ is a field satisfying that \rm{(i)} either
$\fn{char}k =0$ or $\fn{char}k = p > 0$ with $p \nmid n!$, and
\rm{(ii)} $\sqrt{-1}\in k$. Then $k(G_n)$ is $k$-rational for
$n\ge 3$.
\end{theorem}

\begin{proof}
Step 1. Apply Theorem \ref{t2.2}. We find that $k(G_n)$ is
rational over $k(x_i,y_i: 1\le i\le n)^{G_n}$ where $G_n$ acts on
the rational function field $k(x_i,y_i: 1\le i\le n)$ by Formula
\eqref{eq5.1}.

\bigskip
Step 2. Define $u_0=\sum_{1\le i\le n}x_i$, $v_0=\sum_{1\le i\le
n}y_i$ and $u_i=x_i/u_0$, $v_i=y_i/v_0$ for $1\le i\le n$. Note
that $k(x_i,y_i: 1\le i\le n)=k(u_j,v_j:0\le j\le n)$ with the
relations $\sum_{1\le i\le n}u_i=\sum_{1\le i\le n}v_i=1$. The
action of $G_n$ is given by
\begin{align*}
t:{}& u_0\mapsto -u_0,~v_0\mapsto -v_0,~u_i\mapsto u_i,~v_i\mapsto v_i, \\
\tau:{}& u_0\mapsto u_0,~ v_0\mapsto v_0,~ u_i\mapsto u_{\tau(i)},~ v_i\mapsto v_{\sigma^{-1}\tau\sigma(i)}, \\
\bar{\sigma}:{}& u_0\mapsto v_0\mapsto -v_0,~ u_i\mapsto v_i\mapsto u_i
\end{align*}
where $1\le i\le n$ and $t$, $\tau$, $\bar{\sigma}$ are defined in Formula \eqref{eq5.1}.

Define $w_1=u_0v_0$, $w_2=u_0/v_0$.
Then $k(u_j,v_j: 0\le j\le n)^{\langle t\rangle}=k(u_i,v_i:1\le i\le n)(w_1,w_2)$.

Note that $\tau(w_i)=w_i$ for $1\le i\le 2$, $\bar{\sigma}(w_1)=-w_1$, $\bar{\sigma}(w_2)=-1/w_2$.
By Theorem \ref{t2.1},
$k(u_i,v_i:1\le i\le n)(w_1,w_2)^{G_n/\langle t\rangle}=k(u_i,v_i: 1\le i\le n)(w_2)^{G_n/\langle t\rangle}(w')$
for some $w'$ fixed by the action of $G_n/\langle t\rangle$.
Moreover, we may identify $G_n/\langle t\rangle$ with $S_n$ and identify $\bar{\sigma}$ (modulo $\langle t\rangle$) with $\sigma$.

Define $U_i=u_i - (1/n), V_i=v_i - (1/n)$ for $ 1 \le i \le n$. We
find that $\sum_{1\le i\le n}U_i=\sum_{1\le i\le n} V_i=0$ and the
action of $S_n$ on $k(U_i,V_i:1\le i\le n)$ becomes linear. We
will consider $k(U_i,V_i:1\le i\le n)(w_2)^{S_n}$. The action of
$S_n$ is given by
\begin{equation} \label{eq5.2}
\begin{aligned}
\tau:{}& U_i \mapsto U_{\tau(i)},~ V_i\mapsto V_{\sigma^{-1}\tau\sigma(i)},~ w_2 \mapsto w_2\\
\sigma:{}& U_i\mapsto V_i\mapsto U_i,~ w_2 \mapsto -1/w_2
\end{aligned}
\end{equation}
where $1\le i\le n$, $\tau\in A_n$, $\sigma=(1,2)$ and $\sum_{1\le
i\le n} U_i=\sum_{1\le i\le n} V_i=0$.

\bigskip
Step 3. When $\fn{char} k =0$, there is a short-cut to prove this
theorem. The proof of the general case when $\fn{char} k \nmid n!$
will be postponed till Step 5.

Consider the action of $S_n$ on the linear space $\sum_{1\le i\le
n} k \cdot U_i \oplus \sum_{1\le i\le n} k \cdot V_i$. We will
prove that the representation of $S_n$ associated to this linear
space is reducible.

Let $W =\sum_{1\le i\le n} k \cdot s_i$ be the standard
representation of $S_n$, i.e. $\sum_{1\le i\le n} s_i=0$ and
$\lambda (s_i)=s_{\lambda (i)}$ for all $\lambda \in S_n$, for all
$1 \le i \le n$. Let $W'$ be the representation space of the
tensor product of the standard representation and the linear
character $\varepsilon_n: S_n\to \{\pm 1\}$. We will show that the
representation associated to $\sum_{1\le i\le n} k \cdot U_i
\oplus \sum_{1\le i\le n} k \cdot V_i$ is equivalent to that of $W
\oplus W'$.

Since $\fn{char} k =0$, it suffices to show that the characters of
these two representations are completely the same. This fact is
easy to check for even permutations of $S_n$. As to the odd
permutations, note that every odd permutation can be written as
$\sigma \tau$ for some $\tau \in A_n$. Since $\sigma \tau
(U_i)=V_{\tau(i)}, \sigma \tau (V_i)=U_{\sigma^{-1} \tau \sigma(i)
}$, we find that the value of the character of $\sigma \tau$ for
the representation associated to $\sum_{1\le i\le n} k \cdot U_i
\oplus \sum_{1\le i\le n} k \cdot V_i$ is zero. Hence the result.

\bigskip
Step 4. In this paragraph we consider the general case when
$\fn{char} k \nmid n!$. Since $\sqrt{-1}\in k$, define
$w=(\sqrt{-1}-w_2)/(\sqrt{-1}+w_2)$. We find that $\tau(w)=w$ for
$\tau\in A_n$ and $\sigma(w)=-w$. Apply Theorem \ref{t2.1}. We
find that $k(u_i,v_i:1\le i\le n)(w_2)^{S_n}=k(U_i,V_i:1\le i\le
n)^{S_n} (w'')$ for some $w''$ fixed by the action of $S_n$.

In particular, when $\fn{char} k =0$, apply Theorem \ref{t2.1} to
$k(U_i,V_i:1\le i\le n)^{S_n}$. We find that $k(U_i,V_i:1\le i\le
n)^{S_n}= k(W \oplus W')^{S_n}= k(s_i:1\le i\le n-1)^{S_n}(t_1,
\cdots t_{n-1})$ where each $t_i$ is fixed by $S_n$. Obviously the
field $k(s_i:1\le i\le n-1)^{S_n}$ is $k$-rational. Hence the
result.

\bigskip
Step 5. Now return to the general case when $\fn{char} k \nmid
n!$. Because of the above step, it suffices to show that
$k(U_i,V_i:1\le i\le n)^{S_n}(X_1,\ldots,X_N)$ is $k$-rational
where $N=2\cdot (n!)-2(n-1)$. Once we  know $k(U_i,V_i:1\le i\le
n)^{S_n}(X_1,\ldots,X_N)$ is $k$-rational, we find that $k(G_n)$
is $k$-rational.

Since $\fn{char} k \nmid n!$, we may embed the space $\sum_{1\le
i\le n} k \cdot U_i \oplus \sum_{1\le i\le n} k \cdot V_i$ into
the regular representation space $\oplus_{g\in S_n} k\cdot x(g)$.
Thus Theorem \ref{t2.2} is applicable. We find that $k(x(g):g\in
S_n)^{S_n}$ is rational over $k(U_i,V_i: 1\le i\le n)^{S_n}$.
Explicitly $k(x(g):g\in S_n)^{S_n}=k(U_i,V_i:1\le i\le
n)^{S_n}(Y_1,\ldots,Y_{N'})$ where $N'=n!-2(n-1)$.

On the other hand, the regular representation space $\oplus_{g\in
S_n} k\cdot x(g)$ contains the ordinary permutation action
$\oplus_{1\le i\le n}k\cdot z_i$ where $S_n$ acts on
$z_1,\ldots,z_n$ by $g\cdot z_i=z_{g(i)}$ for $1\le i\le n$, $g\in
S_n$. Apply Theorem \ref{t2.2} again. We find that $k(x(g):g\in
S_n)^{S_n}$ is rational over $k(z_i: 1\le i\le n)^{S_n}$. Since
$k(z_i: 1\le i\le n)^{S_n}$ is $k$-rational, we find that
$k(x(g):g\in S_n)^{S_n}$ is $k$-rational. Hence $k(U_i,V_i:1\le
i\le n)^{S_n}(X_1,\ldots,X_N)$ is $k$-rational. Thus $k(G_n)$ is
$k$-rational.
\end{proof}

The first part of the following theorem was proved by Plans
\cite[Theorem 14, (b)]{Pl2}; there he assumed that $\fn{char} k
=0$.

\begin{theorem} \label{t5.3}
\rm{(1)} If $k$ is a field with $\fn{char}k\ne 2$ or $3$, then
$k(G_3)$ is $k$-rational.

\rm{(2)} If $k$ is a field with $\fn{char}k\ne 2$ or $3$, then
$k(G_4)$ is $k$-rational. Moreover, if $\fn{char}k = 0$, then
$k(G_5)$ is also $k$-rational.
\end{theorem}

\begin{proof}
Case 1. $n=3$.

By Step 2 in the proof of Theorem \ref{t5.2} (note that the
assumption $\sqrt{-1}\in k$ is used only till Step 4 there), it
suffices to consider $k(U_i, V_i : 1 \le i \le 3)(w_2)^{S_3}$
where $\sum_{1\le i\le 3} U_i=\sum_{1\le i\le 3} V_i=0$. Define
$\tau =(1, 2, 3) \in S_3$. The actions are given as below,
\begin{align*}
\tau:{}& U_1\mapsto U_2\mapsto -U_1-U_2,~ V_2\mapsto V_1 \mapsto -V_1-V_2, \\
\sigma:{}& U_1\leftrightarrow V_1,~ U_2\leftrightarrow V_2.
\end{align*}

Define $w_3=U_1/V_2$, $w_4=U_2/V_1$, $w_5=V_1/V_2$. It follows
that $k(U_i,V_i:1\le i\le 3)(w_2)^{S_3}=k(w_j: 2\le j\le
5)(V_1)^{S_3}=k(w_j:2\le j\le 5)^{S_3}(w_0)$ for some $w_0$ by
Theorem \ref{t2.1}.

It remains to show that $k(w_j: 2\le j\le 5)^{S_3}$ is $k$-rational.
Note that
\begin{align*}
\tau:{}& w_2\mapsto w_2,~ w_3\mapsto w_4\mapsto (w_3+w_4w_5)/(1+w_5). \\
\sigma:{}& w_2\mapsto -1/w_2,~ w_3\mapsto 1/w_4,~w_4\mapsto 1/w_3,~ w_5\mapsto w_3/(w_4w_5).
\end{align*}

Define $w_6=(w_3+w_4w_5)/(1+w_5)$.
Note that $k(w_3,w_4,w_5)=k(w_3,w_4,w_6)$ and
\begin{align*}
\tau:{}& w_3\mapsto w_4\mapsto w_6 \mapsto w_3, \\
\sigma:{}& w_6 \mapsto 1/w_6.
\end{align*}

Define $w_7=(1-w_3)/(1+w_3)$, $w_8=(1-w_4)/(1+w_4)$, $w_9=(1-w_6)/(1+w_6)$.
Then $k(w_3,w_4,w_6)=k(w_7,w_8,w_9)$ and $\tau: w_7\mapsto w_8\mapsto w_9\mapsto w_7$,
$\sigma:w_7\mapsto -w_8$, $w_8\mapsto -w_7$, $w_9\mapsto -w_9$.

By Theorem \ref{t2.4} we find that $k(w_2,w_3,w_4,w_5)^{\langle\tau\rangle}=k(w_2,X_1,X_2,X_3)$ where
$X_1=w_7+w_8+w_9$ and
\begin{align*}
X_2 &= \frac{w_7^2w_8+w_8^2w_9+w_9^2w_7-3w_7w_8w_9}{w_7^2+w_8^2+w_9^2-w_7w_8-w_7w_9-w_8w_9}, \\
X_3 &= \frac{w_7w_8^2+w_8w_9^2+w_9w_7^2-3w_7w_8w_9}{w_7^2+w_8^2+w_9^2-w_7w_8-w_7w_9-w_8w_9}.
\end{align*}

Moreover, the action of $\sigma$ is given by
\[
\sigma: w_2\mapsto -1/w_2,~X_1\mapsto -X_1,~ X_2\mapsto -X_3,~X_3\mapsto -X_2.
\]

Apply Theorem \ref{t2.2}.
We find that $k(w_2,X_1,X_2,X_3)^{\langle\sigma\rangle}=k(w_2)^{\langle\sigma\rangle}(Y_1,Y_2,Y_3)$
for some $Y_1$, $Y_2$, $Y_3$ fixed by $\sigma$.
Since $k(w_2)^{\langle\sigma\rangle}$ is $k$-rational,
it follows that $k(w_2,X_1,\break X_2,X_3)^{\langle\sigma\rangle}$ is $k$-rational.

\bigskip
Case 2. $n=4$.

Once again we use Step 2 in the proof of Theorem \ref{t5.2}. It
suffices to consider $k(U_i, V_i : 1 \le i \le 4)(w_2)^{S_4}$
where $\sum_{1\le i\le 4} U_i=\sum_{1\le i\le 4} V_i=0$. Denote by
$\lambda_1=(1,2)(3,4),\lambda_2=(1,3)(2,4), \rho=(1,2,3)$ and
$\sigma=(1,2)$ as before. Then $S_4$ is generated by $\lambda_1,
\lambda_2, \rho, \sigma$.

Define $t_1=U_1+U_2, t_2=V_1+V_2, t_3=U_1+U_3,
t_4=V_2+V_3,t_5=U_2+U_3,t_6=V_1+V_3$. The action of $S_4$ is given
as follows,
\begin{align*}
\lambda_1:{}& t_1\mapsto t_1,~t_2\mapsto t_2,~ t_3\mapsto -t_3,
~t_4 \mapsto -t_4,~ t_5\mapsto -t_5,~t_6 \mapsto -t_6,\\
\lambda_2:{}& t_1\mapsto -t_1,~t_2\mapsto -t_2,~ t_3\mapsto t_3,
~t_4 \mapsto t_4,~ t_5\mapsto -t_5,~t_6 \mapsto -t_6, \\
\rho:{}& t_1\mapsto t_5\mapsto t_3\mapsto t_1,
~t_2 \mapsto t_6\mapsto t_4 \mapsto t_2, \\
\sigma:{}& t_1\leftrightarrow t_2,~ t_3\leftrightarrow t_6,~
t_4\leftrightarrow t_5.
\end{align*}

It follows that $k(t_i: 1 \le i \le
6)(w_2)^{<\lambda_1,\lambda_2>}=k(T_i: 1 \le i \le 6)(w_2)$ where
$T_1=t_1/t_2, T_2=t_3/t_4, T_3=t_5/t_6,
T_4=t_2t_6/t_4,T_5=t_4t_6/t_2,T_6=t_2t_4/t_6$. Moreover, the
actions of $\rho$ and $\sigma$ are given as,
\begin{align*}
\rho:{}& T_1\mapsto T_3\mapsto T_2\mapsto T_1,
~T_4 \mapsto T_5\mapsto T_6 \mapsto T_4, \\
\sigma:{}& T_1\mapsto 1/T_1, ~T_2\mapsto 1/T_3,~T_3\mapsto 1/T_2,
\\&T_4 \mapsto (T_1T_2/T_3)T_6,~ T_5\mapsto (T_2T_3/T_1)T_5, ~T_6
\mapsto (T_1T_3/T_2)T_4.
\end{align*}

By Theorem \ref{t2.2}, it suffices to show that $k(T_i: 1 \le i
\le 3)(w_2)^{<\rho,\sigma>}$ is $k$-rational.

Define
$w_3=(1-T_1)/(1+T_1),w_4=(1-T_2)/(1+T_2),w_5=(1-T_3)/(1+T_3)$.
Then we find
\begin{align*}
\rho:{}& w_2\mapsto w_2,
~w_3 \mapsto w_5\mapsto w_4 \mapsto w_3, \\
\sigma:{}& w_2\mapsto -1/w_2, ~w_3 \mapsto -w_3, ~w_4 \mapsto
-w_5,~w_5 \mapsto -w_4.
\end{align*}

Use Theorem \ref{t2.4} to find $k(T_i: 1 \le i \le
3)(w_2)^{<\rho>}$. The remaining part of the proof is very similar
to the last part of Case 1. The details are omitted.

\bigskip
Case 3. $n=5$.

By \cite[Theorem 11]{Pl2}, $k(G_5)$ is rational over $k(G_4)$.
Since $k(G_4)$ is $k$-rational by Case 2, we are done.

\end{proof}

\newpage
\renewcommand{\refname}{\centering{References}}

\end{document}